\documentclass[12pt,letter]{amsart}
\usepackage[russian,english]{babel}
\usepackage{amsmath,amssymb,latexsym}
\usepackage{comment}
\excludecomment{comment}
\usepackage{graphicx}
\usepackage{amsthm} 
 \usepackage{color}
 
\usepackage[breaklinks]{hyperref}
\hypersetup{
unicode=true,
colorlinks=true,
linkcolor=blue,
citecolor=blue,
urlcolor=blue,
filecolor=blue,
bookmarksnumbered=true,
pdfstartview=FitH,
pdfhighlight=/N
}
\newtheorem{theoremMar}{Markov theorem (1884)\!\!}
\newtheorem{theoremSta}{Stahl’s theorem (1986)\!\!}
\newtheorem{theorem}{Theorem}
\newtheorem{proposition}{Proposition}

\theoremstyle{definition}
\newtheorem{definition}{Definition}
\newtheorem{remark}{Remark}
\newtheorem{conjecture}{Conjecture}
\let\savedef=\endproof
\def\endproof{~$\square$\savedef}

\newcommand{\field}[1]{\mathbb{#1}}

\newcommand{\RR}{\field{R}}

\newcommand{\NN}{\field{N}}
\newcommand{\CC}{\field{C}}

\newcommand{\PP}{\field{P}}

\newcommand{\wt}{\widetilde}

\newcommand{\mc}{\mathcal}

\newcommand{\supp}{\operatorname{supp}}
\renewcommand{\(}{\left( }
\renewcommand{\)}{\right) }

\def\GRS{\mathrm{GRS}}

\let\myt\widetilde
\let\myo\overline

\let\geq\geqslant

\let\leq\leqslant

\begin{document}

\selectlanguage{english}

\title{Tschebyshev--Pad\'e approximations for multivalued functions}

\author[E.\,A.~Rakhmanov]{Evguenii~A.~Rakhmanov}
\address{University of South Florida, USA}
\email{rakhmano@mail.usf.edu}
\author[Sergey~P.~Suetin]{Sergey~P.~Suetin}
\address{Steklov Mathematical Institute of the Russian Academy of Sciences, Russia}
\email{suetin@mi-ras.ru}


\maketitle

\begin{abstract}

We discuss  the relation between the linear Tschebyshev--Pad\'e approximations to analytic function $f$ and the diagonal type I Hermite--Pad\'e polynomials for the tuple of functions $[1,f_1,f_2]$ where the pair of functions $f_1,f_2$ forms certain Nikishin system. 
 An approach is proposed of how to extend the seminal  Stahl's Theory for Pad\'e approximations for multivalued analytic functions to the Tschebyshev--Pad\'e approximations. The approach is based on the relation between Tschebyshev--Pad\'e  approximations and Hermite--Pad\'e polynomials and also on a connection of  Hermite--Pad\'e polynomials and multipoint Pad\'e approximants.

Bibliography:~\cite{Tre20} titles.

\end{abstract}

\setcounter{tocdepth}{1}\tableofcontents

\section{Tschebyshev--Pad\'e Approximations}\label{s1}
Tschebyshev--Pad\'e approximations are rational (with free poles) approximations to orthogonal expansions. Theory of free poles rational approximations is a classical fields on the boundary of approximation theory and complex analysis. At the same time it is one of the hot topics in contemporary mathematics due to many new applications and connections with the other problems from various branches of classical analysis. In particular, convergence problems for Tschebyshev--Pad\'e approximation for function with branch points have many important connections and some of them are discussed in this paper.

On the other hand, constructions of free poles rational approximations are known to be numerically effective and widely used in applications (see~\cite{BaGr96},~\cite{Boy01},~\cite{Tre13},~\cite{AuTr17}). In the end of the paper we present results of some numerical experiments illustrating our theoretical exposition.


\subsection{Definitions.} \label{s11}
Let $\mu$ be a positive measure on the interval $\Delta:=[-1,1]$
and $T_n(x;\mu)=\varkappa_n(\mu)x^n+\dotsb$ be the associated sequence of orthonormal polynomials
\begin{equation}
\int_{\Delta} T_k(x;\mu)T_n(x;\mu)\,d\mu(x)=\delta_{kn},\quad k,n=0,1,2,\dots,
\label{1.1}
\end{equation}
where $\delta_{kn}$ is the Kronecker symbol and $\varkappa_n(\mu)>0$.

Any function $f \in L^2(\mu)$ 
may be represented by a series with respect to the system $\{T_k(x;\mu)\}_{k=0}^\infty$:
\begin{equation}\label {OS}
f(z)=\sum_{k=0}^\infty c_kT_k(z;\mu),\quad c_k=c_k(f)
=\int_{\Delta} f(x)T_k(x;\mu)\,d\mu(x).
\end{equation}
which converges to $f(z)$ in $L^2(\mu)$. If $f$ is holomorphic on $\Delta$ and $\mu'(x):=d\mu(x)/dx >0 $ a.e. on ${\Delta}$ then series \eqref{OS} converges inside (on the compact subsets) of the maximal 
ellipse with foci $\pm 1$ where $f$ is holomorphic (see~\cite{Sze75}). The series diverges outside of this ellipse. A classical approach to the problem of the analytic continuation of $f$ to larger domains is based on rational approximation to expansions \eqref{OS} with free poles. A typical ``Pad\'e - like'' construction (first explicitly introduced in \cite{ClLo74}, see also~\cite{ChCo81},~\cite{Fle73}) is the following.

Let $\PP_n$ be the set of all algebraic polynomials with complex coefficients of degree $\leq{n}$. The (linear) Tschebyshev--Pad\'e approximant to $f$ of order $n$ is the ratio of two polynomials $\Phi_n(z) = P_n/Q_n$ from $\PP_n$ defined by

\begin{equation}
(Q_nf - P_n)(z)=\sum_{k=2n+1}^\infty a_{n,k}T_k(z;\mu).
\label{1.3}
\end{equation}
The rational function $\Phi_n(z)$ is called the $n$\,th diagonal Frobenious--Pad\'e (or linear Tschebyshev--Pad\'e) approximant to the series~\eqref{OS}. Since $a_{n,k}=c_k(Q_nf-P_n)$ and $P_n\in\PP_n$, the relations~\eqref{1.3} are equivalent to
\begin{equation}
c_k(Q_nf)=0,\quad k=n+1,\dots,2n.
\label{1.4}
\end{equation}
In turn, equations~\eqref{1.4} may be written as a system of $n$ linear homogeneous equations for $n+1$ unknowns coefficients of the polynomial $Q_n\in\PP_n$. Such system has a nontrivial solution so that approximations $\Phi_n=P_n/Q_n$ always exists. They may not be unique unlike to Pad\'e approximants to a power series. For uniqueness problem see~\cite{Ibr02} and Remark~\ref{rem1} below.

It is easy to see that to find the polynomial $Q_n$ from the system~\eqref{1.4} we need $3n+1$ coefficients $c_k(f)$, $k=0,1,\dots,3n$, of the series~\eqref{OS} should be given.


If denominator $Q_n$ is known, then corresponding numerator $P_n$ is determine by
\begin{equation}
P_n(z)=\sum\limits_{k=0}^n c_k(Q_nf)T_k(z;\mu).
\label{1.42}
\end{equation}



\begin{remark}\label{rem1}
There is another version of definition called Baker definition~\cite{BaGr96}
(see first of all~\cite{Hol69} and also~\cite{GoRaSu91},~\cite{GoRaSu92},~\cite{GoRaSu11}).
According to this definition the rational function $F_n$ of order $\leq n$ is defined by
\begin{equation}
(f - F_n)(x)=\sum_{k=2n+1}^\infty b_{n,k}T_k(x;\mu).
\label{1.5}
\end{equation}
Here only $2n+1$ coefficients $c_0,\dots,c_{2n}$ of the power series~\eqref{OS} are needed to construct $F_n$.
Such rational function $F_n$ does not always exist; see~\cite{Sue81} and~\cite{Sue09}. If $F_n$
exists it is called also the diagonal nonlinear Pad\'e approximation to the polynomial series~\eqref{OS}. In general $F_n\neq \Phi_n$. In this paper we mainly consider linear approximations $\Phi_n$.
\end{remark}


\subsection{Convergence problems and plan of the paper.} \label{s11}
First convergence result for Tschebyshev--Pad\'e approximations were obtained in the papers~\cite{GoRaSu91} and~\cite{GoRaSu92} where  for the Markov type function $f(z)=\widehat{\sigma}(z)$, where
\begin{equation}
\widehat{\sigma}(z)=\int_c^d\frac{d\sigma(t)}{z-t},\quad z\notin[c,d],
\quad S_\sigma = [c,d]\subset\mc R\setminus{\Delta}.
\label{1.6}
\end{equation}
Both constructions $\Phi_n$ and $F_n$ has been studied for such functions $f$. It was proven that the rational functions $\Phi_n$ are always unique while functions $F_n$ always exist. Convergence of both sequences $\{\Phi_n(z)\}$ and $\{F_n(z)\}$ to $f(z)$ inside the domain $\myo\CC\setminus[c,d]$ was proven and the comparison of the corresponding rates of convergence was given. It seems that similar results are not known for any other general classes of functions 



In the current paper we consider convergence problems for linear Tschebyshev--Pad\'e (Tschebyshev--Pad\'e--Frobenious) approximations $\Phi_n$ to the orthogonal series~\eqref{OS} for functions with finite number of branch points; class of such functions we will denote by $\mc A$.

In case of classical diagonal Pad\'e approximants at infinity a basic convergence theorem for the class $f = \widehat{\sigma}$ has been proved by A. Markov in 1884 and untill 1985 this theorem had been essentially the only general result in the field (together with associated Stieltjes theorem for the case when $\supp \sigma $ is not a compact set). A century later, in 1985 a convergence theorem for the class $\mc A$ has been proved by H. Stahl (see~\cite{Sta97b}  and the references therein) and it was arguably one of the most significant event in the theory in 20-th century. In case of Tschebyshev--Pad\'e approximants the analogue of 
the Markov theorem is valid. The convergence problem for class $\mc A$ is open and its solution does not seem to be close. We will try to outline an approach based on a certain far reaching generalization of an original Stahl's method.    

Method originated by Stahl has been further developed in~\cite{GoRa87}. The generalized version (which we call GRS-method) has a larger circle of applications; solutions of a number of problems has been obtained in ~\cite{GoRa87} and subsequent papers  \cite{Bus13},  \cite{Bus15},  \cite{Bus21 }, \cite{Rak12}, [R-h-Pade]. These results strongly suggest that the method has a potential to be further generalized but any such generalization will require significant efforts. In particular, the GRS-method is not yet developed far enough for applications to Hermite--Pad\'e approximations.

It turns out that convergence problems for Tschebyshev--Pad\'e approximants may be reduced to zero distribution problems for the first kind Hermite--Pad\'e approximants for a system of three functions  constituting a Nikishin system. Moreover, it seems to be the only visible general approach to the problem. 
It follows from the remarks above that this approach is rather problematic itself. There are no known results which can be used and we have to begin with a study of certain basic  problems for Hermite--Pad\'e approximants.

In turn, a natural approach to zero distribution problem for the Hermite--Pad\'e approximants may be based on a general interpolation problem which presents also an independent interest.  In the next sections we will discuss a situation briefly outlined above in some details. We will begin with a brief review of original Markov and Stahl results.




\section{Pad\'e approximants at infinity}\label{s2}
Given a function analytic at infinity or formal series in powers of $1/z$
\begin{equation}
f(z)=\sum_{n=0}^\infty \frac {c_n}{z^n}
\label{PS1}
\end{equation}
we can find for any $n \in \NN$ polynomials $P_n, Q_n \in P_n$ such that
\begin{equation}
\bigl(Q_n f - P_n\bigr)(z)=O\(\frac1{z^{n+1}}\),
\quad z\to\infty,
\label{PA}
\end{equation}
and $Q_n$ is not identical zero. The ratio $P_n/Q_n$ for any pair of such polynomials is unique and is called the $n$-th (diagonal) Pad\'e approximants to the series $f$.

There are two specially important classes of functions for which there is a general convergence theorem. The first one is an old simple classical result which was originally stated in terms of continued fractions.

\begin{theoremMar}
Let $f(z) = \widehat{\sigma}(z)$ where $\sigma$ is a positive measure on the real line with a compact support. Then sequence of Pad\'e approximants  (PA) $P_n/Q_n$ converges to $f$ uniformly on compact sets in $\myo\CC \setminus S$ where $S$ is the minimal segment of $\RR$ containing $S_\sigma$.
\end{theoremMar}

If $S = S_\sigma$ consists of a finite number of intervals $\sigma'(x) >0$ almost everywhere (a.e.) on $S$ then the exact rate of convergence may be also indicated. Both convergence and its rate may be obtained from zero distribution for denominators $Q_n$
\begin{equation}
\frac1n\chi(Q_{n})\to \lambda,\quad n\to\infty,
\label{ZD1}
\end{equation}
where $\lambda$ is the Robin measure of $S$.

The proof is easily obtained from the fact the Pad\'e denominators $Q_n$ are orthogonal polynomials with respect to $\sigma$
\begin{equation}
\int x^k Q_n(x) d\sigma(x) =0, \quad k = 0,1 ,\dots, n-1.
\label{Ort0}
\end{equation}
In particular, zeros of such polynomials (poles of approximations) all belong the minimal segment of $\RR$ containing $S_\sigma$.






\subsection{Stahl's convergence theorems}\label{s2s1}
Stahl's convergence theorem for PA \eqref{PA} at infinity is related to multivalued functions with a small set of branch points. To state the theorem we have to introduce few important definitions.

Let $\mc D$ be a domain in the extended plane $\myo\CC.$ We denote by $H(\mc D)$ space of holomorphic (singlevalued) functions in $\mc D.$ We further denote by $\mc A(\mc D)$ class of multivalued analytic functions in $\mc D$; more exactly, we write $f \in \mc A(\mc D)$ if $f$ is determined by an elements at some point $z_0 \in \mc D$ (or in a subdomain of $\mc D$) which allow for analytic continuation along any path in $\mc D$ which begins at $z_0$.

Let $e \subset\CC$ be a finite set and function $f$ defined by a series \eqref{PS1} at $\infty$ belongs to
$ \mc A(\myo\CC \setminus e)$. We denote by $\mc F$ the set of cuts $F$ which make $f$ singlevalued, that is
\begin{equation}
\mc F = \mc F_f = \{F \subset\CC: \ f\in H(\CC \setminus F) \}
\label{Cut1}
\end{equation}
(we also assume that $F \in \mc F$ satisfy certain additional assumptions which may vary depending on a current situation; for a moment this is not important).

\begin{definition}\label{def0}
Finally, we say that $\Gamma \in \mc F$ satisfies $S$-property if the following conditions are satisfied

(a) $\Gamma$ is a finite union of analytic arcs and may be a compact set of zero capacity,

(b) in the interior points $\zeta \in\Gamma^0$ of these arcs we have the following symmetry condition
\begin{equation}\label{GRS2}
\frac {\partial U^\lambda}{\partial n_1}(\zeta) = \frac {\partial U^\lambda}{\partial n_2}(\zeta), \quad
\zeta \in \Gamma^0,
\end{equation}
where $U^\lambda$ is the potential of the equilibrium (Robin) measure $\lambda$ of $\Gamma$ and $n_{1,2}$ are oppositely oriented normal vectors to $\Gamma$ at $\zeta$,

(c) jump $(f^+ - f^-)(\zeta)$ of $f$ across $\Gamma$ is not identical zero on any analytic arc in $\Gamma$.

We call a compact set with these properties an $S$-compact set.

\end{definition}

With these definitions we have the following

\begin{theoremSta}
Let $f \in \mc A(\myo\CC \setminus e)$ where $e$ is (here) a finite set.
Then

(i) there exists a unique $S$-compact set $\Gamma \in \mc F_f$,

(ii) sequence $P_n/Q_n$ of Pad\'e approximants for $f$ converges to $f$ in capacity on compact sets in $\myo\CC \setminus \Gamma$.
\end{theoremSta}

Proof of the theorem is essentially more complicated than the proof of Markov theorem. In the part (i) related to existence of $S$-compact the method is rather traditional. It was known (see~\cite{Nut84}) that
$S$-compact set $\Gamma$ minimizes capacity in class $\mc F$. Minimal capacity problem was known in geometric function theory and its version related to the situation has been solved by Stahl in frames of this theory.

The essential difference between the two cases is in the orthogonality conditions for denominators $Q_n$. For $x \in \RR$ we have $\myo x = x$ and therefore orthogonality relations \eqref{Ort0} present orthogonality in a Hilbert space. Pad\'e denominators $Q_n$ in Stahl theorem satisfy nonhermitian orthogonality relations
\begin{equation}
\oint_{\Gamma} z^k Q_n(z) f(z)\,dz = 0, \quad k = 0, 1, \dots, n-1,
\label{Ort1}
\end{equation}
which surprisingly share certain properties of hermitian orthogonal polynomials but it is rather difficult to prove.


\subsection{Generalizations}\label{s2s1}
A study of Tschebyshev--Pad\'e approximants may be based on an analogue of Stahl's theorem for type I  Hermite--Pad\'e polynomials for two multivalued functions. This includes two problems.

The first on is related to the construction of the proper generalization of the notion of $S$-compact set. The idea of how to solve this problem was proposed in the papers ~\cite{Sue18c}, \cite{Sue15}, \cite{Sue18},~\cite{Sue19}, ~\cite{IkSu20}. The new approach to the existence problem developed in these papers is based on consideration of an equilibrium problem on a double-sheeted Riemann surface.
Characterization of the $S$-compact set associated with the problem may be given by different terms; we go into some details below.

A more challenging problem is that nonhermitian type orthogonality conditions \eqref{Ort1} in case of Hermite--Pad\'e approximants will turn into a combination of two systems of orthogonality relations. How to work with such combination is known in some cases but the case under consideration is new and require a new method. We go into some details next.


\section{Hermite--Pad\'e polynomials.}\label{s2}

Let functions $f_1, f_2$ determined by series at infinity
\begin{equation}
f_j(z)=\sum_{n=0}^\infty \frac {c_{n,j}}{z^n}, \quad j = 1,2
\label{PS2}
\end{equation}
Corresponding type I Hermite--Pad\'e polynomials $Q_{n,0}\in\PP_{n-1}$ and $Q_{n,1}, Q_{n, 2}\in\PP_n$ are defined by the following relation
\begin{equation}
R_n(z) = \bigl(Q_{n, 0}+Q_{n,1}f_1+Q_{n,2}f_2\bigr)(z)=O\(\frac1{z^{2n+2}}\),
\quad z\to\infty
\label{Ort2}
\end{equation}
(called also HP-polynomials for the tuple of three functions $[1,f_1,f_2]$ and multiindex $n=(n-1,n,n)$;
$R_n$ is associated remainder).
Such polynomials $Q_{n,j}$ exist but may not unique (this will not be essential here; see~\cite[Chapter~4]{NiSo88} and~\cite{StRjDr20}).

But for the functions given by representations~\eqref{2.1} the tuple $[Q_{ n,0},Q_{n,1},Q_{n,2}]$ is uniquely determined up to a nontrivial constant multiplier.

\subsection{Case of two Markov functions}

Let the measures $\mu$ and $\sigma$ are two positive measure, $S_\mu = \Delta = [-1, 1]$ and $S_\sigma = [c, d] \subset \RR \setminus \Delta$. We define
\begin{equation}
f_1(z)=\widehat{\mu}(z):=\int_{\Delta}\frac{d\mu(x)}{z-x},\quad
f_2(z)=\langle\mu,\sigma\rangle(z):=\int_{\Delta}\frac{\widehat{\sigma}(x)\,d\mu(x)}{z-x},\quad z\in D,
\label{2.1}
\end{equation}
where $D:=\myo\CC\setminus\Delta$ and $\widehat{\sigma}(z)$ is given by~\eqref{1.6}.

Corresponding type I Hermite--Pad\'e polynomials $Q_{n,0}\in\PP_{n-1}$ and $Q_{n,1}, Q_{n, 2}\in\PP_n$ are uniquely determined by~\eqref{1.6} up to a nontrivial constant multiplier.

In this case we have explicit formula for the $n$-th order Tschebyshev--Pad\'e approximant $\Phi_n =P_n/Q_n$ for $f(z)=\widehat{\sigma}$ (see \eqref{1.3} in terms of Hermite--Pad\'e polynomials $Q_{n,j}$ for $(f_1. f_2)$ in \eqref{Ort2}

\begin{proposition}\label{pro1}
We have
\begin{equation}
\Phi_n(z) = \frac {P_n(z)}{Q_n(z)} = - \frac {Q_{n, 1}(z)}{Q_{n, 2}(z)}
\label{Red}
\end{equation}
\end{proposition}
\begin{proof}
It follows directly from definitions that 
\begin{equation}
\int_\gamma R_n(t)q(t)\,dt=0\quad\text{for each}\quad q\in\PP_{2n},
\label{2.3}
\end{equation}
where $\gamma$ is a closed curve that separates the set ${\Delta}$ from infinity. 
Since the term with $Q_{n,0}$ in $R_n$ vanishes and taking also into account representations ~\eqref{2.1} we obtain
\begin{equation}
\int_{\Delta}(Q_{n,1}+Q_{n,2}\widehat{\sigma})(x)q(x)\,d\mu(x)=0,
\label{2.4}
\end{equation}
where $q\in\PP_{2n}$ is an arbitrary polynomial. From here, since $f = \widehat{\sigma}$ we have
\begin{equation}
(Q_{n,1}+Q_{n,2}f)(z)=\sum_{k=2n+1}^\infty c_{n,k}T_k(z;\mu).
\label{2.6}
\end{equation}
By uniqueness property of Tschebyshev--Pad\'e(--Frobenious) approximation~\cite{GoRaSu91} $\Phi_n$ to $f(z)=\widehat{\sigma}(z)$, we have that $\Phi_n=- Q_{n,1}/Q_{n,2}$ and Proposition~\ref{pro1} follows.
\end{proof}


 \begin{remark}\label{rem1}
From Proposition~\ref{pro1} it follows that the diagonal Tschebyshev--Pad\'e approximations can be found from the Laurent expansions at the infinity point $z=\infty$ of the functions $f_1(z)$ and $f_2(z)$ given by~\eqref{2.1}. Clearly the coefficients of those Laurent expansions can be found directly from the coefficients $c_k$ of the series~\eqref{OS} and vice versa. Thus for computing of the rational function $\Phi_n=-Q_{n,1}/Q_{n,2}$ it is possible to use the algorithms for computing the type I Hermite--Pad\'e polynomials for the tuple $[1,f_1,f_2]$ and multiindex $n=(n-1,n,n)$; see~\cite{IkSu21},~\cite{IkSuHePa21} and the bibliography therein.

Note that both Tschebyshev--Pad\'e approximations and Hermite--Pad\'e polynomials are constructive in the sense of P.~Henrici paper~\cite[Sec.~2]{Hen66} and are based on just the same initial data (see also~\cite{Tre20}).
\end{remark}

It follows from \eqref{2.4} that for $f = \widehat{\sigma}$ the combination \eqref{2.6}
\begin{equation}
Q_{n,1}+Q_{n,2}f = Q_n f - P_n,
\end{equation}
which is related to Hermite--Pad\'e polynomials $Q_{n,1}, Q_{n,2}$ for the pair $f_1, / f_2$ in \eqref{2.1} and simultaneously to Tschebyshev--Pad\'e approximants $\Phi_n = P_n/Q_n$ has $2n - 1$ zeros on the interval $\Delta$. Let us denote the monic polynomial with these zeros by $\Omega_n$.

Zero distribution of polynomials $\Omega_n$ and $Q_n = Q_{n,2}$ are defined by the following vector equilibrium problem. Let $\mc M(\Delta, \Delta_1)$ be the set of positive vector measures $\vec \mu = (\mu_0, \mu_1)$ where $\mu_0$ is a positive measure on $\Delta$ with total mass $\mu_0(\Delta) = 2$ and $\mu_1$ is a positive measure on $F=[c,d]$ with total mass $\mu_1(F) = 1$. The energy $\mc E$ of the vector-measure $(\mu_0, \mu_1)$ is defined by
\begin{equation}\label{E1}
\mc E(\mu_0, \mu_1) = \int U^{\mu_0}(x)\,d \mu_0(x)-\int U^{\mu_0}(t)\,d\mu_1(t)+ \int U^{\mu_1}(t)\,\mu_1(t),
\end{equation}
where $U^\mu (z) = \displaystyle-\int\log |z -x| d\mu(x)$  is the logarithmic potential of $\mu$. Equivalently, this energy is defined by positive definite $2\times 2$-matrix
$M = \begin{pmatrix} 1 & -1/2 \\ -1/2 & 1 \end{pmatrix}$.

There is a unique pair of measures $(\lambda_0, \lambda_1) \in \mc M(\Delta, F)$ 
which minimizes vector-energy
\begin{equation}\label{E2}
\mc E(\lambda_0, \lambda_1) = \min_{\mc M} \mc E(\mu_0, \mu_1)
\end{equation}
where $\mc M = \mc M(F_0, F_1)$ \\

\begin{proposition}\label{pro2}
Under the same assumptions \eqref{2.1} as in Proposition~\ref{pro1} we have as $n \to \infty$
$$
\frac1n\chi(Q_{n,1}), \to \lambda_1 \quad \frac1n\chi(Q_{n,2})\to\ \lambda_1 \quad
\quad\text{and}\quad \quad
\frac1n\chi(\Omega_n)\to \lambda_0,\quad
$$
\end{proposition}

For a proof and further details see~\cite{GoRaSu91},~\cite{GoRaSu92}). Convergence $\Phi_n(z) \to f(z)$ together with the rate follows from Proposition~\ref{pro2}.

\subsection{General case}\label{s2s2}
In case when
\begin{equation}
f_j(z)=\sum_{n=0}^\infty \frac {c_{n,j}}{z^n} \in A(\myo\CC \setminus e_j) \quad j = 1,2
\label{PS2}
\end{equation}
with arbitrary finite sets $e_1$, $e_2$ we do not have even a reliable conjecture on the zero distribution of Hermite--Pad\'e polynomials \eqref{Ort2}. There is a few isolated results in~\cite{Nut84},~\cite{Apt08},~\cite{ApKuAs08},~\cite{RaSu13},~\cite{MaRaSu16} where some details and remarks may be found (see also~\cite{KoPaSuCh17}).

The convergence problem for Tschebyshev--Pad\'e approximants which is our main concern is related to the so-called Nikishin case when functions $f_1, f_2$ have common branch point $\pm 1$ such that ratio
\begin{equation}
f(x) = \frac {f_2^+(x) - f_2^-(x) }{f_1^+(x) - f_1^-(x) }, \quad x \in (-1, 1),
\label{f}
\end{equation}
of the jumps of functions $f_1$ and $f_2$ across $\Delta$ is extended as analytic function in class $A(\myo\CC \setminus e)$ with a finite set $e \subset \myo \CC \setminus \Delta$. Observe that for functions $f_1(z)$ and $f_2(z)$ in \eqref{2.1} and $f(z) = \widehat \sigma(z)$ we have exactly this relation. In the sequel we restrict ourselves to this case.

A representative example of such functions is given by
\begin{equation}
{f}(z):= 
\prod_{j=1}^m\(A_j-\frac1{\varphi(z)}\)^{\alpha_j},
\label{L}
\end{equation}
where $|A_j|>1$, $\sum\limits_{j=0}^m \alpha_j = 0$ and $\zeta = \varphi(z):=z+(z^2-1)^{1/2}$ be the inverse to Zhukovskii function $z = (\zeta +1/\zeta)/2$.

In particular, for $m=2$ the Tschebyshev--Pad\'e approximation $\Phi_n$ to such function admits the following representation in terms of type I Hermite--Pad\'{e} polynomials
$$
\Phi_n(z)=\frac{Q_{n,1}(z)}{Q_{n,2}(z)}+\frac1{\sqrt{A_1A_2}}.
$$
This class of function has been introduced in~\cite{Sue18b} where the further details may be found.

For functions $f$ of the form \eqref{L} (and for more general class which we describe later) we will formulate a general conjecture on convergence of Tschebyshev--Pad\'e approximants. We will suggest also a method of proof based on a study of multipoint Pade  approximants to functions of class $\mc A$. This is a general free poles interpolation problem which also present an independent interest.

\section{General interpolation problem}

Let $e \subset \myo \CC$ be a finite set and $f \in \mc A(\myo\CC \setminus e)$. We denote by
\begin{equation}
\mc F = \mc F_f = \{F \subset\CC: \ f\in H(\CC \setminus F)\}
\label{Cut}
\end{equation}
the set of cuts $F \in \mc F$ which make $f$ singlevalued. We also assume (here)  that $F \in \mc F$ is a finite union of smooth (analytic) arcs (additional assumptions on compacts $F \in \mc F$ may vary; see~\cite{RaSu13});\\



Let $\PP_n$ denote the set of all polynomials of degree at most $n$ and $\Omega_n \in\PP_{2n}$. For any function $f$ holomorphic in an open set containing the set $Z_{\Omega_n}$ of zeros of $\Omega_n$ there exist a pair of polynomials $P_n \in \PP_{n-1}$ and $Q_n \in \PP_n$ such that
\begin{equation}\label{Int}
\frac {(Q_n f - P_n)(z)}{\Omega_n(z)} \in H(Z_{\Omega_n})
\end{equation}
and $Q_n$ is not an identical zero. Conditions \eqref{Int} are called linear interpolation conditions; they imply that $R_n = Q_nf - P_n$ vanishes at zeros of $\Omega_n$ (together with a number of derivatives in case of multiple zeros). These conditions are reduced to a linear system of $2n$ homogeneous equations with $2n+1$ unknown coefficients of $P_n, Q_n$ which has a nontrivial solution.

Rational function $P_n/Q_n$ generally interpolate $f$ at zeros of $\Omega$ but, may be, not at all zeros.
It is possible that $\Omega_n(\zeta) = 0$ for some $\zeta \in Z_{\Omega_n}$ and at the same time
$P_n(\zeta) = Q_n(\zeta) = 0$ so that $(P_n/Q_n)(\zeta) \neq f(\zeta)$. In such case we will say that interpolation condition at $\zeta$ is lost. Clearly that if $k_n$ is the number of lost interpolation then in fact
$\deg P_n \leq n-1 - k_n$ and $\deg Q_n \leq n - k_n$ in the representation $\Phi_n=P_n/Q_n$.

Assume that the sequence $\{\Omega_n \}_{n \in \NN}$ (interpolation table) has a weak-$*$ limit distribution 
\begin{equation}\label{Lim}
\frac 1n \chi(\Omega_n) = \frac 1n \sum_{\Omega( \zeta) = 0}\delta(\zeta)\overset {*}{\to} 2\omega
\end{equation}
where $\omega$ is a unit positive measure. We will consider convergence problem for the sequence of interpolating functions $P_n/Q_n$, i.e.  multipoint Pad\'e approximants associated with an interpolation table $\{\Omega_n \}$ and also closely related problem of zero distribution of denominators~$Q_n$.


\begin{remark}\label{rem11}
In case when some or all zeros of $\Omega_n$ are located at infinity interpolation conditions \eqref{Int} has to be modified. In particular, if $Z_{\Omega_n} = \{\infty\}$ (with multiplicity $2n$) interpolation conditions \eqref{Int} are written as $(Q_n f - P_n)(z) = O(1/z^{n+1})$ (assuming $f(\infty) = 0$) so that $P_n/Q_n$ are classical Pad\'e approximants at $\infty$
\end{remark}

Convergence problem for the sequence $P_n/Q_n$ is one of the central problems in the classical approximation theory and there is a large number of results in this direction. 
In the connection with convergence of Tschebyshev--Pad\'e approximants we need to consider the interpolation problem under more general assumptions on the limiting density of the interpolation table than in all the known results. In such generality the problem is open and in this connection we will present a conjecture. As an introduction to this conjecture we consider the case of interpolation in a finite number of points. This case is not in not an immediate cotollary of known results too but it may be essentially investigated in frames of GRS method. 

\subsection{Orthogonality conditions for denominators $Q_n$}
The convergence problem for multipoint Pad\'e approximants $P_n/Q_n$ is essentially reduced to a study of zero distribution of Pad\'e denominators $Q_n$. These polynomilas may also be viewed as nonhermitian orthogonal polynomials with certain variable weights. The following proposition is well known starting point of GRS method.

\begin{proposition}\label{pro11}
Let $f(\infty) = 0$, $F \in \mc F_f$, $Z_{\Omega_n}\cap F = \varnothing$ and $\deg \Omega_n = 2n$ Then
\begin{equation}
\oint_{F} Q_n(z) z^k \frac{f(z)\,dz}{\Omega_n(z)} = 0, \quad k = 0, 1, \dots, n-1,
\label{Ort}
\end{equation}
where $\displaystyle\oint_{F}$ means integration over the boundary of the domain $\myo\CC \setminus F$ which moves around nonintegrable singularities of $f$ along arcs of small circles in $\myo\CC \setminus F$ apart from points in $Z_{\Omega_n}$.
\end{proposition}

\begin{proof}
We have $R_n(z) = (Q_n f - P_n)(z)/\Omega_n(z) \in H(\myo\CC \setminus F)$ and this function has zero of order $\geq n +1$ at infinity. Therefore $z^k R_n(z) \in H(\myo\CC \setminus F)$ and has zero of order $\geq 2$ at infinity if $k \leq n-1$. Now \eqref{Ort} follows by the Cauchy integral theorem since them relate to integration of $P_n/\Omega_n$ vanishes.
\end{proof}

\section{$\GRS$ theorem}
Under certain assumption on the limit density $\omega$ the zero distribution of the denominators $Q_n$ defined in \eqref{Ort} may be obtained directly from $\GRS$-theorem which we formulate below as Theorem~\ref{the1}. Conditions of this theorem are more general in the part related to a variable (depending on $n$) weights: the weight $1/\Omega_n$ in orthogonality conditions \eqref{Ort} is replaced by an abstract analytic weight function $\Psi_n(z)$ satisfying, however, a uniform convergence condition
\begin{equation}\label{GRS1}
\frac 1n \log|\Psi_n(z)| \to \psi(z), \quad z \in \Gamma,
\end{equation}
where $\psi(z)$ is harmonic in a domain containing $\Gamma \in \mc F$ which satisfies $S$-property in the external field $\psi$, i.e. (here) $\Gamma$ is a finite union of analytic arcs and in the interior points $\zeta \in\Gamma^0$ of this arcs we have the symmetry condition
\begin{equation}\label{GRS2}
\frac \partial{\partial n_1} \bigl( U^\lambda +\psi\bigr)(\zeta) =
\frac \partial{\partial n_2} \bigl(U^\lambda +\psi\bigr)(\zeta)
\end{equation}
($U^\lambda$ is the potential of the equilibrium measure $\lambda$ in the external field $\psi$ on $\Gamma$ and $n_{1,2}$ are oppositely oriented normal vectors to $\Gamma$ at $\zeta$). It is also important that jump $(f^+ - f^-)(\zeta)$ of $f$ across $\Gamma$ is not identical zero on any analytic arc in $\Gamma$. Under these conditions the following theorem is valid; see~\cite{GoRa87}.

\begin{theorem}\label{the1}
If polynomials $Q_n \in \PP_n$ are defined by orthogonality relations
\begin{equation}
\oint_{\Gamma} Q_n(z) z^k \Psi_n(z) f(z)\,dz = 0, \quad k = 0, 1, \dots, n-1,
\label{Ort11}
\end{equation}
and conditions \eqref{GRS1} and \eqref{GRS2} are satisfied, then $\frac 1n \chi(Q_n) \overset {*}{\to} \lambda$ as $n \to \infty$.
\end{theorem}

\subsection{Existence of $S$-compact set}
The main challenge in applications of the theorem is the existence of an $S$-compact set $\Gamma \in \mc F_f$ related to a situation. In most cases such a compact exists and unique but there are exceptions and it is often not easy to establish fact of existence. In some cases existence may be proved by max-min energy problem (see~\cite{Sta88},~\cite{PeRa94},~\cite{MaRaSu11},~\cite{Sta12}, \cite{Rak12},~\cite{RaSu13},~\cite{MaRa16}) which we discuss later in connection with our current problem \eqref{Ort}. Anyway, situation is essentially simplified if existence of an $S$-compact set is known.

An immediate corollary of Theorem~\ref{the1} for multipoint Pad\'e approximants to a function $f \in \mc A(\myo\CC \setminus e)$ (see \eqref{Cut}) under assumption \eqref{Lim} is the following. If there is a compact set
$\Gamma \in \mc F_f$ with $S$-property in the external field $\psi = -U^\omega$ which is homotopic to $F = F_f$ in $\myo\CC \setminus(e \cup Z_{\Omega_n})$ then orthogonality relations \eqref{Ort} may be written with $\Gamma$ in place of $F$.

In this case Theorem~\ref{the1} together with Proposition~\ref{pro1} implies that $\frac 1n \nu(Q_n) \overset {*}{\to} \lambda$ as $n \to \infty$ where $\lambda = \lambda_\psi$ is the unit equilibrium measure on $\Gamma$ in the field $\psi = -U^\omega$. Convergence $P_n/Q_n \to f$ (in capacity) on compact sets in complement
to $\Gamma$ normally follows.

Situation when there is no $S$-compact set (compact set $\Gamma$ with $S$-property) in class $\mc F_f$
is more complicated. $\GRS$-theorem may not directly applied and we have to go into all the details of the situation. In particular it is important to figure out geometric properties of the $S$-configuration which will allow to apply $\GRS$-theorem or its modification.


\subsection{Max-min energy method in frames of general interpolation}
In the connection to multipoint Pad\'e approximants \eqref{Int} we can make an immediate progress for the case when the support of the limit density $\omega$ is a finite set (see~\cite{Bus13},~\cite{Bus15},~\cite{Bus21})
\begin{equation}
\omega = \sum_{k=1}^p m_k\delta(\beta_k)\quad\text{where}\quad \sum_{k=1}^p m_k = 1,
\quad b_k \in\myo\CC \setminus F.
\label{ome-1}
\end{equation}
Recall that we interpolate branch $f \in H(\myo\CC \setminus F)$ of a function $f \in \mc A(\myo\CC \setminus e)$ determined by a cut $F\in \mc F_f$. Orthogonality conditions \eqref{Ort} allow for continuous deformation of $F$ in the domain of analyticity of $f/\Omega_n$ which means, in particular, that class $\mc F_f$ is not adequate to the situation; to preserve conditions \eqref{Ort} we need homotopy in the domain of analyticity of $f/\Omega_n$.

If there is no $S$-compact set homotopic to $F$ in the domain of analyticity of $f/\Omega_n$ we have to consider a larger class $\wt {\mc F}$ integration contours such that orthogonality conditions \eqref{Ort} is preserved. Together with continuous deformation of $F$ in the domain of analyticity of $f/\Omega_n$ we have to consider deformations which produce homologous (but not homotopic) contour of integration
(a curve jumps over a point $\beta \in \supp \omega$). The class of curves $\wt {\mc F}$ which we which is obtained this way depends on both $f \in \mc A(\myo\CC \setminus e) $ and sequence $\Omega_n$.


\begin{proposition}\label{pro12}
There exist $\wt \Gamma \in\wt {\mc F}$ with $S$-property in the external field $\psi = - U^\omega$.
\end{proposition}

The max-min proof of the proposition is based on the functional of equilibrium energy. Let $\mc E_\psi (\mu)$ be the energy of the measure $\mu$ in the external field $\psi = - U^\omega$. For $F \in \wt {\mc F}$ we denote by $\lambda_F$ the associated unit equilibrium measure on $F$ which is defined by the condition that $\mc E_\psi (\mu)$ takes its minimal value in class of unit positive measures $\mu$ on $F$ only for $\mu = \lambda_F$.

Next we maximize equilibrium energy $\mc E_\psi (\lambda_F)$ over $F \in \wt {\mc F}$. If the maximizing compact set $\wt \Gamma \in \wt {\mc F} $
\begin{equation}\label{M-M}
\mc E_\psi (\lambda_{\wt\Gamma}) = \max_{F \in\wt {\mc F}} \mc E_\psi (\lambda_F)
\end{equation}
exists then we use local variations to prove that it has the $S$ property in the external field $\psi = - U^\omega$ (cf.~\cite{PeRa94},~\cite{MaRaSu11},~\cite{RaSu13}).


As a corollary we obtain a theorem on zero distribution of polynomilas $Q_n$ (we do not go into further duscussions of details related this theorem). 

\begin{theorem}\label{the2}
Suppose that we interpolate $f \in\mathcal A(\myo \CC \setminus e)$ using a table $\Omega_n$ with limiting density $\omega$ (see \eqref{Lim}) and measure $\omega$ has a finite support \eqref{ome-1} (assume also that zeros of polynomials $\Omega_n$ belong to the support of $\omega$).

Let $\lambda$ b the unit equilibrium measure on $\wt \Gamma \in \wt {\mc F} $ in the external field $\psi(z)= -U^\omega(z)$ and $\wt \Gamma$ is from Proposition~\ref{pro12}.

Then the balayage of $\frac 1n \chi(Q_n)$ onto the support of $\lambda$ is weakly convergent to $\lambda$ as $n \to \infty$. 
\end{theorem}


\begin{remark}\label{rem3}
Condition that zeros of polynomials $\Omega_n$ belong to the support of $\omega$ is not essential. There is only $o(n)$ zeros of $\Omega_n$ outside of arbitrarily small neighbourhood of $\beta \in\supp \omega$. Any single zero $\zeta$ of $\Omega_n$ is an obstacle for deformation of $F$ ``through'' $\zeta$ if we want to preserve orthogonality conditions \eqref{Ort}. However, such a deformation is possible if we eliminate interpolating conditions \eqref{Int} by setting
$P_n(\zeta) = Q_n(\zeta) = 0$. It seems that this a common way of elimination of a small number of the ``wrong'' interpolation conditions (interpolation of a ``wrong'' branch). Loss of $o(n)$ of interpolation conditions does not influence zero distributions.
\end{remark}

\subsection{Quadratic differential}
Set $\wt \Gamma$ may be described in terms of trajectories of quadratic differentials in a way similar to what we have in case of classical Pad\'e approximants. This description is obtained as an intermediate result in application of max-min method. Let $e = \{a_1, a_2, \dots, a_m\}$ and $A(z) = \prod_{j=1}^m (z -a_j)$ and the limit measure $\omega$ for interpolation table. Then there exist a polynomial $V(z) =z^{m+2p - 3}$ such that $\wt \Gamma$ is a union of some of (critical) trajectories of the quadratic differential
\begin{equation}\label{QD}
-\frac{V(z)\, dz^2}{A(z) B(z)^2}\quad\text{where}\quad B(z) = \prod_{k=1}^p (z- \beta_k),
\quad\beta_k \in\CC \setminus F
\end{equation}
There are $p$ equations involving residues of the function $\sqrt{V/A}/B$ at points $z = \beta_k$ on parameters (roots or coefficients) of $V$. Using this equations we can study in details case $A(x) = z^2 -1$, $B(z) = z^2 + 1 - \epsilon$ with different values of $m_1$ and $m_2$.

\subsection{Conjecture~\ref{conj1}}
Theorem~\ref{the2} above is a direct corollary of known results on max-min energy and $\GRS$-method. Under more general assumptions on $\omega$ the method may be applied after some modifications. Some of these modifications are not immediate and at this moment we do not have formal proof of the theorem generalizing Theorem~\ref{the2} and we will state a conjecture for a general case.

Moreover, it is not yet obvious that the conjecture below is, indeed, valid without more restrictions on $\omega$.

\begin{conjecture}\label{conj1}
Assume that $F \in \mc F_f$ and interpolation table $\{\Omega_n\}$ do not intersect $F.$ Assume further that limit density $\omega$ in \eqref{Lim} exists and $\supp\omega$ do not intersect $e$. Then there exist a unique compact set $\wt \Gamma$ with max-min property \eqref{M-M} and $\frac 1n \nu(Q_n) \overset {*}{\to} \lambda_{\wt \Gamma}$ as $n \to \infty.$
\end{conjecture}

The max-min compact set $\wt \Gamma$ has $S$-property but the definition of this property has to be slightly modified.

\section{Conjecture on Tschebyshev--Pad\'e approximants}

Let $\mc R$ be the Riemann surface of the function $\sqrt{z^2 -1}$ defined as a double-sheeted (branched) covering of $\myo\CC$ with quadratic branch points $\pm 1$. Let $\pi\colon \mc R \to \myo \CC$ be associated canonic projection. Lifting $\mc R \setminus \pi^{-1}(\Delta)$ of the  domain $\myo\CC\setminus\Delta = [-1, 1]$ onto $\mc R$ is union of two disjoint domains $\mc D^0$ and $\mc D^1$ both projected by $\pi$ one-to-one onto $ \myo \CC \setminus \Delta$.

Let $e = \{a_1, a_2, \dots, a_m\}$ be distinct point in $\CC \setminus \Delta.$ Together with points $a_j \in e$ we will have to consider their images $a'_j = \pi^{-1}(a_j) \in \mc D^1$.
Thus, set $e' = \{a'_1, a'_2, \dots, a'_m\}$ belongs to the first sheet $\mc D^1 \subset \mc R$.

Let function $\widehat f$ on $\mc R$ be holomorphic in $\mc D^0$ and has analytic continuation from $\mc D^0$ along any path on $\mc R$ which does not hit any point from $e'$ (at least two points from $e'$ are branch points for $\widehat f$.

This function has a continuous extension to the boundary of $\mc D^0$ which is the union of two copies of the interval $\Delta$ whose endpoints are identified. We select arbitrarily any of these intervals and denote $f(z) = \widehat f (\pi^{-1}(z))$. Function $f(z)$ is continuous on $\Delta$ an holomorphic on $(-1, 1)$.

Analytic continuation of $f$ from $(-1,1)$ to the upper half-plane may be extended to a function $f \in H(\myo\CC \setminus \Delta)$. Lifting of this extension is one of two function either $\widehat f|_{\mc D^0}$ or
$\widehat f|_{\mc D^1}$. We will agree that continuation of $f$ to upper half plane leads to the projection $f _0$ of the function $\widehat f|_{\mc D^0}$ (this makes the choice of the branch of $f$ unique).
Now, analytic continuation of $f$ to the lower half-plane may extended to a function $f_1(z)$ holomorphic in the complement to $e \cup \Delta$.

We will consider Tschebyshev--Pad\'e approximation to the function $f(z)$ with an analytic weight $w(x)$ on $\Delta$. The approximant of order $n$ is the ratio of two polynomials $P_n/Q_n$ of degree at most $n$ which are defined by
\begin{equation}
\int_{\Delta} x^k \,(Q_n f - P_n)(x) \, \frac {w(x) dx}{\sqrt{1-x^2}} = 0, \quad k = 0, 1, \dots, 2n-1.
\label{Ort2.1}
\end{equation}

Integration over $\Delta$ in this definition may be replaced with integration over a smooth curve $F_0$ connecting points $-1$ and $1$ in a domain of analyticity of $w$ which is homotopic to $\Delta$ in $\CC \setminus e$ (we assume that $w$ may be extended far enough; exact conditions will be formulated below). Set of such curves will be denoted by $\mc F_0$

Next, we have $f_1 \in \mc A(\myo\CC \setminus (e \cup \Delta))$. Let $\mc F_1$ be the set of admissible cuts for $f_1$, that is, for any $F_1 \in \mc F_1$ we have $f_1 \in H(\myo\CC \setminus F_1)$. We assume also that $F_1 \in \mc F_1$ is a finite union of smooth arcs.

We will consider only pairs $\vec {\mc F} =\{ (F_0, F_1)\}$ of disjoint compacts $F_k \in \mc F_k$ ($k =0, 1$). The description of behaviour of approximants $P_n/Q_n$ defined by \eqref{Ort2.1} is given in terms of a vector $S$-compact set $(\Gamma_0, \Gamma_1) \in \vec {\mc F}$ which may be defined as follows.

Let $\mc M(F_0, F_1)$ be the set of positive vector measures $\vec \mu = (\mu_0, \mu_1)$ on $(F_0, F_1)$ respectively with masses $\mu_0(F_0) = 2$ and $\mu_0(F_1) = 1$. The energy $\mc E$ of the vector-measure $(\mu_0, \mu_1)$ is defined by
\begin{equation}\label{E1}
\mc E(\mu_0, \mu_1) = \int U^{\mu_0}d\mu_0 - \int U^{\mu_0}d\mu_1 + \int U^{\mu_1}d\mu_1.
\end{equation}
Equivalently, this energy is defined by $2\times 2$ positive definite matrix
$M = \begin{pmatrix} 1 & -1/2 \\ -1/2 & 1 \end{pmatrix}$

For any $(F_0, F_1)$ there is a pair of measures $(\lambda_0, \lambda_1) \in\mc M(F_0, F_1)$ on $F_0$ and $F_1$ respectively which solve the matrix equilibrium problem with masses $\lambda_0(F_0) = 2$ and $\lambda_0(F_1) = 1$ and vector equilibrium measure $(\lambda_0, \lambda_1)$ minimizes vector-energy
\begin{equation}\label{E2}
\mc E(\lambda_0, \lambda_1) = \min_{\mc M} \mc E(\mu_0, \mu_1),
\end{equation}
where $\mc M = \mc M(F_0, F_1)$.

\begin{proposition}\label{pro14}
There is a pair of compact sets $\Gamma_0 \in \mc F_0$ and $\Gamma_1 \in \mc F_1$ which maximizes equilibrium energy $\mc E(\lambda_0, \lambda_1)$.
\end{proposition}

Denote
$(\lambda^\Gamma_0, \lambda^\Gamma_1)$ components of the equilibrium measure, $\Gamma=(\Gamma_0,\Gamma_1)$.

The idea of the proof of the Proposition~\ref{pro14} was proposed in~\cite{Sue19}; it is based on the equivalent reformulation of the equilibrium problem \eqref{E1}--\eqref{E2} in terms of weighted potential on the Riemann surface $\mc R$ (see~\cite{Chi18},~\cite{Chi20}). Descriptions of the equilibrium vector-compact $(\Gamma_0, \Gamma_1)$ may be given using different terms.

\begin{conjecture}\label{conj2}
Let $P_n/Q_n$ be $n$-th order Tschebyshev--Pad\'e approximant to the function $f_0(z)$. Then we have, first, $\frac 1n \chi(Q_n) \overset {*}{\to} \lambda_1^\Gamma$ as $n \to \infty$. Second, function $R_n = Q_n f - P_n$ has $2n(1+ o(1))$ zeros around $\Gamma_0$ and the normalized counting measure $\sigma_n$ of these zeros weakly converges to $\lambda^\Gamma_0$.
\end{conjecture}

A proof of the conjecture may be based on Conjecture~\ref{conj1}. We outline a plan of such proof assuming that Conjecture~\ref{conj1} is proved in complete generality.

We may need Conjecture~\ref{conj1} in maximal generality since we do not have an {\it a'priory} information on location of zeros of the function $R_n = Q_n f - P_n$ but we know, at least, that we have to consider branch $f = f_1$ of this function in $D_1 = \myo \CC \setminus \Gamma_1$. Assume that this function has $2n(1+ o(1))$ zeros in $D_1$. We denote by $\Omega_n$ polynomial with zeros at these points which belong to $D_1$. At this moment we will ignore zeros of $R_n$ on $\Gamma_1$.

If there are too many of them on $\Gamma_1$ we may have to pass to a slightly different branch of $f$ but now we do not go into such details.

Then we select a weakly convergent subsequences: we find a subsequence $\Lambda \subset \NN$ such that as $n \to \infty, \ n\in \Lambda$, \eqref{Lim} is valid and at the same time $\frac 1n \nu(Q_n) \overset {*}{\to} \eta$. Now we have to prove that $(\omega, \eta) = (\lambda^\Gamma_0, \lambda^\Gamma_1)$.

We consider interpolation problem for the branch of $f_1$ in $\myo\CC \setminus \Gamma_1 $ and derive from Conjecture~\ref{conj1} that $\eta$ is the equilibrium measure on $\Gamma_1$ in the external field
$-U^\omega$.

It follows also that after a proper normalization we have convergence in capacity
\begin{equation}\label{R}
\frac 1n \log|R_n(z)| \to U^\eta(z) - U^\omega(z).
\end{equation}
Then orthogonality conditions \eqref{Ort2.1} imply that $\omega$ is the equilibrium measure on $\Gamma_0$ in the external field $-{}^1\!/_2 U^\eta$. This would conclude the proof.

\section{One numerical example}\label{s7}

In this Section we discuss a numerical example on the distribution of the zeros of Hermite--Pad\'e polynomials and poles and zeros of Tschebyshev--Pad\'e approximations. All the numerical computations were performed using the Program HePa.com~\cite{IkSuHePa21}. This Program is based on a generalization of the classical Viskovatov algorithm~\cite{IkSu21}.

In this section we denote Pad\'e polynomials of degree $\leq{n}$ by $P_{n,0}$ and $P_{n,1}$, i.e., for a given $\mathfrak f\in H(\infty)$
$$
(P_{n,0}+P_{n,1}\mathfrak f)(z)=O\(\frac1{z^{n+1}}\),\quad z\to\infty.
$$
Denominator and numerator of Tschebyshev--Pad\'e approximation $\Phi_n(z)$ of order $n$ for a series $\mathfrak h$ of type~\eqref{OS} is denoted here by $Q_n$ and $P_n$ respectively, i.e. $\Phi_n(z)=P_n(z)/Q_n(z)$ and
$$
(Q_n\mathfrak h-P_n)(z)=\sum_{k=2n+1}^\infty c_{n,k}T_k(z;\mu).
$$
As before, type I Hermite--Pad\'e polynomials for multiindex $(n,n,n)$ and for a tuple of functions $[1,\mathfrak f_1,\mathfrak f_2]$, $\mathfrak f_1,\mathfrak f_2\in H(\infty)$, is denoted by $Q_{n,0},Q_{n,1}$ and $Q_{n,2}$, i.e.
$$
(Q_{n.0}+Q_{n,1}\mathfrak f_1+Q_{n,2}\mathfrak f_2)(z)=
O\(\frac1{z^{2n+2}}\),\quad z\to\infty.
$$

Let function $f$ be from the class given by the explicit representation~\eqref{L}. Here we set $m=3$, $\alpha_1=\alpha_2=-1/3$, $\alpha_3=2/3$ and $A_3\in(1,2)$, $A_{1,2}=e\pm ib$, where $e>0$ is a small real number. Thus
\begin{equation}
 f(z)=\biggl(A_1-\frac1{\varphi(z)}\biggr)^{-1/3}
\biggl(A_2-\frac1{\varphi(z)}\biggr)^{-1/3}
\biggl(A_3-\frac1{\varphi(z)}\biggr)^{2/3}
\label{Z2}
\end{equation}
(here $\varphi(z)=z+(z^2-1)^{1/2}\sim 2z$ as $z\to\infty$).
Let function $f$ be given by~\eqref{Z2}. Let consider two tuples $[1, f,f^2]$ and $[1,1/(z^2-1)^{1/2},f]$. Clearly that the pair $(-Q_{n,1},Q_{n,2})$ from  the second tuple gives a Tschebyshev--Pad\'e approximation to the function $\mathfrak h$, i.e. $\Phi_n(z)=-Q_{n,1}(z)/Q_{n,2}(z)$. Thus zeros of $Q_{n,1}$ and $Q_{n,2}$ are respectively zeros and poles of the Tschebyshev--Pad\'e approximation $\Phi_n$.

Let now take into account all that was said before about the connection between type I Hermite--Pad\'e polynomials and Tschebyshev--Pad\'e approximations and about max-min problem which was conjectured to describe the limit zeros distribution of Hermite--Pad\'e polynomials. Then based on that information we might suppose that the limit zero distribution of the corresponding HP-polynomials for both tuples
$[1, f,f^2]$ and $[1,1/(z^2-1)^{1/2},f]$, where $f$ is from~\eqref{Z2},
 should be just the same. The numerical results represented on the Figure~\ref{Fig_hp} and Figure~\ref{Fig_che} are in a good accordance with that statement.

\newpage\clearpage
\begin{figure}[!ht]
\centerline{
\includegraphics[width=15cm,height=15cm]{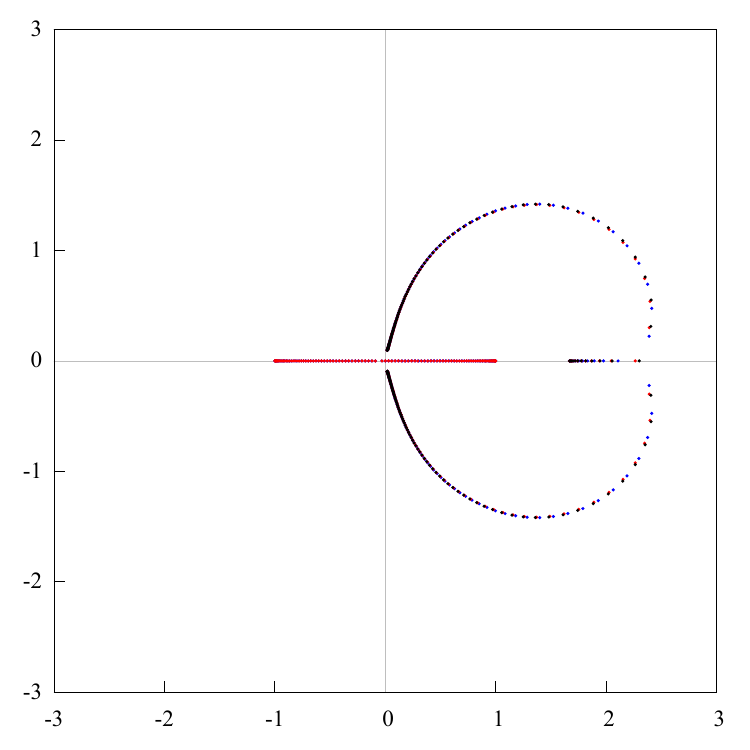}}
\vskip-6mm
\caption{
Here the zeros of Pad\'e polynomials $P_{n,0}(z),P_{n,1}(z)$ of order
$n=100$ for the function $f$ given by~\eqref{Z2}  are plotted  (blue points for $P_{n,0}$ and red points for $P_{n,1}$). These zeros simulate the segment $\Delta=[-1,1]$ which is the Stahl compact set for $f$. Also the zeros of type I HP-polynomials $Q_{n,0}(z), Q_{n,1}(z)$ and $Q_{n,2}(z)$ of order
$n=200$ for the tuple $[1,f,f^2]$ are plotted (blue points for $Q_{n,0}$, red points for $Q_{n,1}$ and black points for $Q_{n,2}$). Those zeros simulate the $S$-compact set $\myt{\Gamma}$ from Proposition~\ref{pro12} which is symmetric with respect to the real line (cf. Fig.~\ref{Fig_che}).	
}
\label{Fig_hp}
\end{figure}
\newpage\clearpage
\begin{figure}[!ht]
\centerline{
\includegraphics[width=15cm,height=15cm]{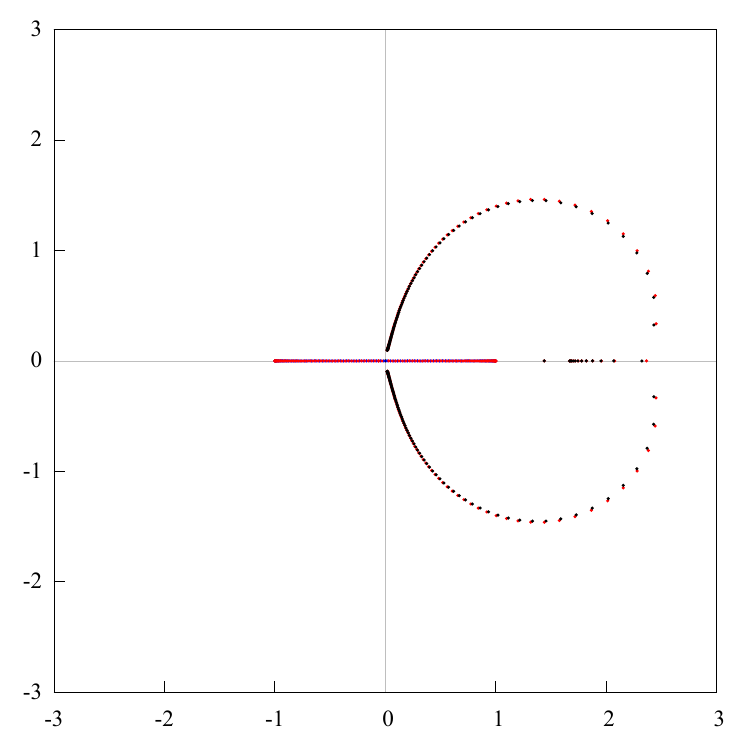}}
\vskip-6mm
\caption{
Here the zeros of Pad\'e polynomials $P_{n,0}(z),P_{n,1}(z)$ of order
$n=100$ for the function $1/(z^2-1)^{1/2}$  are plotted  (blue points for $P_{n,0}$ and red points for $P_{n,1}$). These zeros simulate the segment $\Delta=[-1,1]$ which is the Stahl compact set for the function  $1/(z^2-1)^{1/2}$. Also the zeros of type I HP-polynomials $Q_{n,1}(z)$ and $Q_{n,2}(z)$ of order $n=200$ for the tuple $[1, 1/(z^2-1)^{1/2},f]$ are plotted (red points for $Q_{n,1}$ and black points for $Q_{n,2}$). Those zeros simulate the $S$-compact set $\myt{\Gamma}$ from Proposition~\ref{pro12} which is symmetric with respect to the real line  (cf. Fig.~\ref{Fig_hp})
}
\label{Fig_che}
\end{figure}


\newpage\clearpage

\def\by#1;{#1\unskip,}
\def\paper#1;{``#1\unskip''\unskip,}
\def\paperinfo#1;{#1\unskip.}
\def\eprint#1;{``#1\unskip''\unskip,}
\def\eprintinfo#1;{#1\unskip,}
\def\book#1;{``#1\unskip''\unskip,}
\def\inbook#1;{``#1\unskip''\unskip,}
\def\bookinfo#1;{#1\unskip,}
\def\jour#1;{#1\unskip,}
\def\issue#1;{#1\unskip,}
\def\yr#1;{#1\unskip,}
\def\pages#1.{#1\unskip.}
\def\vol#1;{\textbf{#1}\unskip,}
\def\finalinfo#1;{#1\unskip.}
\def\publ#1;{#1\unskip,}
\def\publadrr#1;{#1\unskip,}
\def\procinfo#1;{#1\unskip,}
\def\serial#1;{#1\unskip,}
\def\ed#1;{ed #1\unskip,}
\def\eds#1;{ed #1\unskip,}


\end{document}